\documentclass[12pt]{amsart}
\usepackage{amsmath}
\usepackage{amsfonts}
\usepackage{amssymb}
\usepackage{amsthm}
\usepackage{eucal}
\newtheorem{theorem}{Theorem}[section]

\usepackage{mathrsfs}
\usepackage{eucal}
\usepackage{graphicx}

\newcommand{\zz}{\mathbb{Z}}

\newcommand{\pp}{\mathbb{P}}
\renewcommand{\gg}{\mathbb{G}}

\newcommand{\E}{\mathcal{E}}
\newcommand{\F}{\mathcal{F}}
\newcommand{\G}{\mathcal{G}}
\renewcommand{\L}{\mathcal{L}}

\newcommand{\N}{\mathcal{N}}

\renewcommand{\S}{\mathcal{S}}

\renewcommand{\O}{\mathscr{O}}

\newcommand{\e}{\vec{e}}
\newcommand{\Sigmabar}{\overline{\Sigma}}
\DeclareMathOperator{\codim}{codim}
\newtheorem{Lemma}[theorem]{Lemma}

\DeclareMathOperator{\spec}{Spec}
\DeclareMathOperator{\Proj}{Proj}
\DeclareMathOperator{\Sym}{Sym}

\DeclareMathOperator{\Pic}{Pic}
\DeclareMathOperator{\rank}{rank}
\DeclareMathOperator{\Quot}{Quot}
\DeclareMathOperator{\Mor}{Mor}
\DeclareMathOperator{\Def}{Def}
\theoremstyle{definition}
\newtheorem{Example}[theorem]{Example}
\theoremstyle{remark}
\newtheorem*{remark}{Remark}
\theoremstyle{remark}
\newtheorem*{warning}{Warning}
\usepackage{tikz}
\usepackage{tikz-cd}
\usepackage{wasysym, stackengine, makebox, graphicx}
\tikzset{
  symbol/.style={
    draw=none,
    every to/.append style={
      edge node={node [sloped, allow upside down, auto=false]{$#1$}}}
  }
}
\usepackage{url}

\tikzset{cong/.style={draw=none,edge node={node [sloped, allow upside down, auto=false]{$\cong$}}},
         Isom/.style={draw=none,every to/.append style={edge node={node [sloped, allow upside down, auto=false]{$\cong$}}}}}

\usetikzlibrary{matrix,arrows,decorations.pathmorphing}
\usepackage[margin=1.2in]{geometry}

\numberwithin{equation}{section}

\title{Universal degeneracy classes for vector bundles on $\pp^1$ bundles}
\author{Hannah K. Larson}
\date{\today}
\begin{document}
\maketitle

\begin{abstract}
Given a vector bundle on a $\pp^1$ bundle, the base is stratified by degeneracy loci measuring the spitting type of the vector bundle restricted to each fiber. 
The classes of these degeneracy loci in the Chow ring or cohomology ring of the base are natural invariants characterizing the degenerations of the vector bundle.
When these degeneracy loci occur in the expected codimension, we find their classes. This yields universal formulas for degeneracy classes in terms of naturally arising vector bundles on the base. Our results hold over arbitrary fields of any characteristic.
\end{abstract}

\section{Introduction}

Vector bundles on families of rational curves arise naturally in many geometric situations.
The Grothendieck-Birkoff Theorem states that any vector bundle $E$ on $\pp^1$ splits as a direct sum of line bundles $E \cong \O_{\pp^1}(e_1) \oplus \cdots \oplus \O_{\pp^1}(e_r)$ for integers $e_1 \leq \cdots \leq e_r$. 
We call a non-decreasing collection of integers $\vec{e} =  (e_1, \ldots, e_r)$ a \textit{splitting type}, and abbreviate the corresponding sum of line bundles by $\O(\vec{e})$.  

This paper studies this splitting phenomenon in families, over arbitrary fields of any characteristic. Let $W$ be a rank $2$ vector bundle on a scheme $B$ and form the $\pp^1$ bundle $\pi: \pp W \rightarrow B$.
Given a vector bundle $E$ on $\pp W$, we define \textit{splitting loci} of $E$ by
\[\Sigma_{\vec{e}}(E) := \{b \in B: E|_{\pi^{-1}(b)} \cong \O(\vec{e})\} \subset B.\]
The \textit{expected codimension} of the splitting locus $\Sigma_{\vec{e}}(E)$ is 
\[u(\vec{e}) := h^1(\pp^1, End(\O(\vec{e}))) = \sum_{i < j} \max\{0, e_j - e_i - 1\},\] 
which is the dimension of the deformation space of the bundle $\O(\vec{e})$. Deformation theory shows (see e.g. \cite[Ch. 14]{EH}) that if $\Sigma_{\vec{e}}(E)$ is non-empty,
\[\codim \Sigma_{\vec{e}}(E) \leq u(\vec{e}).\]


Splitting loci often have geometric significance.
Below are some examples of naturally arising vector bundles on families of rational curves and their splitting loci.

\begin{enumerate}
\item Consider a projective variety $X \subset \pp^n$ and suppose $F$ parametrizes rational curves of a given degree on $X$. The family of curves $\mathcal{C}$ parametrized by $F$ sits inside $X \times F$, and the normal bundle $\N_{\mathcal{C} / X \times F}$ is a family of vector bundles on $\mathcal{C}$ whose splitting loci govern the local geometry of these curves inside $X$.
The splitting of this bundle has been studied extensively for $X=\pp^n$ (see for example  \cite{AR, CR, EV81, EV82, R07, Sa80, Sa82}) and for other varieties (e.g.  \cite{D01, CR2, F14, L, K96}). The present work answers the remark following Proposition 2.3 of \cite{L}, which asked for the classes of splitting loci of the normal bundle of the universal line on the universal hypersurface.
\item Suppose $C$ is a projective curve with a non-constant degree $k$ map $f: C \rightarrow \pp^1$. Let $\L$ be a Poincar\'e line bundle (universal line bundle) on $\Pic^d(C) \times C$. Then $(id \times f)_* \L$ is a family of vector bundles on $\Pic^d(C) \times \pp^1$ whose splitting loci parametrize certain types of linear systems on $C$. The Brill-Noether loci of general $k$-gonal curves (see \cite{JR, Pf}) are certain unions of these splitting loci.
The universal formulas described here are used to show existence of all splitting degeneracy loci expected to occur for a general $k$-gonal curve \cite{L2}.
\item There is no analogous classification result for vector bundles on $\pp^n$ for $n > 1$. One approach to studying vector bundles on $\pp^n$ is by restricting to each line $\pp^1 \subset \pp^n$. For each vector bundle on $\pp^n$, this gives rise to a family of vector bundles on lines whose splitting loci are called the loci of jumping lines and are important geometric invariants of the bundle (see e.g. \cite{Ellia}). More generally, given a variety $X$ with a vector bundle $E$, one may study $E$ through the varieties of jumping curves in the moduli space of maps $\pp^1 \rightarrow X$. When $X = \gg(k, n)$ and $E$ is the tautological bundle, these splitting loci describe the different types of rational scrolls in $\pp^n$ 
(see e.g. \cite{C, LVX, VX}). As another example, \cite{Ca} studies the case when $X$ is the moduli space of stable vector bundles on a curve of genus at least $2$.
\end{enumerate}

We write $\vec{e}\nobreak\hspace{.16667em plus .08333em}' \leq \vec{e}$ if splitting type $\vec{e}$ can specialize to $\vec{e}\nobreak\hspace{.16667em plus .08333em}'$, that is if $e_1' + \ldots + e_k' \leq e_1 + \ldots + e_k$ for all $k$ (see Section \ref{sch}). With this notion, we define \textit{splitting degeneracy loci} (set theoretically) by
\[\overline{\Sigma}_{\vec{e}}(E) := \bigcup_{\vec{e}\nobreak\hspace{.16667em plus .08333em}' \leq \vec{e}} \Sigma_{\vec{e}\nobreak\hspace{.16667em plus .08333em}'}(E).\]
A priori, it is not clear how to construct the ``right" scheme structure on $\overline{\Sigma}_{\vec{e}}(E)$.
This subtlety will be discussed in Section \ref{sub}, where we confirm that the scheme structure we choose has the following minimality property: the tangent space at a point in the open stratum $\Sigma_{\vec{e}}(E) \subset \overline{\Sigma}_{\vec{e}}(E)$ is precisely those maps $\spec k[\epsilon]/(\epsilon^2) \rightarrow B$ so that the induced first order deformation of $E|_{\pi^{-1}(\spec k)}$ is trivial. 

The classes of splitting degeneracy loci are natural invariants characterizing the degenerations of a vector bundle on a $\pp^1$ bundle.
We give a constructive proof that, when splitting loci occur in the correct codimension, these classes are given by a universal formula in terms of the Chern classes of naturally arising vector bundles on the base.
In some cases, the formula is particularly simple.
\begin{Example} \label{-1,1}
Suppose $E$ is a rank $2$, degree $0$ vector bundle on $B \times \pp^1$. On the open set $B \backslash \overline{\Sigma}_{(-2,2)}$, the theorem on cohomology and base change shows that the pushforwards $\pi_*E$ and $\pi_* E(1)$ are locally free sheaves of ranks $2$ and $4$ respectively.
There is a natural map between rank $4$ vector bundles $\phi: \pi_*E \otimes H^0(\O(1)) \rightarrow \pi_* E(1)$, and
\[\overline{\Sigma}_{(-1,1)} = \{b \in B : \rank \phi_b \leq 3\}. \]
If $\codim \overline{\Sigma}_{(-1,1)} = 1$ then the class of $\overline{\Sigma}_{(-1,1)}$ in $B \backslash \overline{\Sigma}_{(-2,2)}$ is given by the Porteous formula (see e.g. \cite[Thm. 12.4]{EH}):
\[[\overline{\Sigma}_{(-1,1)}] = c_1(\pi_*E(1)) - 2c_1(\pi_*E).\]
If $\codim \overline{\Sigma}_{(-2,2)} > 1$, then this formula also holds in the Chow ring of $B$ by excision (see e.g. \cite[Prop. 1.14]{EH}).
\end{Example}
The above example worked because we could compute on the open subset of $B$ where $\pi_* E(1)$ and $\pi_*E$ were locally free, and the splitting locus $(-1, 1)$ was determined by a single rank condition. The latter fails in general (see Example \ref{s2ex}) and new ideas are needed to compute the classes of splitting degeneracy loci in general.

We give a closed formula for degeneracy classes for certain splitting types, and an inductive algorithm that works for all splitting types. 
\begin{theorem} \label{main}
Let $E$ be a vector bundle on a $\pp^1$ bundle $\pp W \rightarrow B$. 
Assume that $\codim \Sigma_{\vec{e}}(E) = u(\vec{e})$ and $\codim \Sigma_{{\vec{e}} \nobreak\hspace{.16667em plus .08333em}'} > u(\vec{e})$ for all $\vec{e}\nobreak\hspace{.16667em plus .08333em}' <\vec{e}$.
The class of $\overline{\Sigma}_{\vec{e}}(E)$  in $A^*(B)$ is given by a universal formula, depending only on $\vec{e}$, in terms of Chern classes of $\pi_*\O_{\pp W}(1), \pi_* E(m),$ and $\pi_*E(m-1)$ for $m$ suitably large.  This formula is computed by the procedure in Section \ref{gen}.
\end{theorem}

\begin{remark}
Even if the dimension of $\overline{\Sigma}_{\vec{e}}(E)$ is larger than expected, the class resulting from this formula is still represented by a cycle supported on $\overline{\Sigma}_{\vec{e}}(E)$. In particular, if the expected class for $\overline{\Sigma}_{\vec{e}}(E)$ is non-zero, then $\overline{\Sigma}_{\vec{e}}(E)$ must be non-empty.
\end{remark}

\begin{remark}
The push forwards $\pi_* E(m)$ and $\pi_* E(m-1)$ for $m$ supplied by the algorithm in Section \ref{gen} will be locally free on a suitably large open subset, but need not be locally free on all of $B$. Using Lemma \ref{ind}, one may express their Chern classes in terms of vector bundles $\pi_* E(i)$ and $\pi_*E(i-1)$ for any $i$ large enough that $R^1\pi_*E(i-1) = 0$.
\end{remark}

\begin{remark}
One can deduce the classes of splitting loci on a general family of genus zero curves $\mathcal{C} \rightarrow B$ up to $2$-torsion by studying the fiber product
 \begin{center}
 \begin{tikzcd}
& \mathcal{C} \times_B \mathcal{C} \arrow{r}{q} \arrow{d}[swap]{p} &\mathcal{C} \arrow{d}{\pi} \\
& \mathcal{C} \arrow{r}[swap]{\pi} & B
 \end{tikzcd}
 \end{center}
The diagonal inside $\mathcal{C} \times_B \mathcal{C} $ is a degree $1$ divisor on each fiber, making $p$ into a $\pp^1$ bundle. Given a vector bundle $E$ on $\mathcal{C}$, we can therefore compute the class of $\overline{\Sigma}_{\vec{e}}(q^* E)$ on $\mathcal{C}$, and this locus is $\pi^{-1} (\overline{\Sigma}_{\vec{e}}(E))$. The relative tangent bundle $T_{\pi}$ of $\mathcal{C} \rightarrow B$ restricts to a degree $2$ line bundle on each fiber. Thus, $\pi_*(c_1(T_{\pi}) \cdot [\overline{\Sigma}_{\vec{e}}(q^* E)]) = 2[\overline{\Sigma}_{\vec{e}}(E)]$.
\end{remark}

This paper is organized as follows. In Section \ref{sch}, we review basic facts about splitting loci and describe an important example. In Section \ref{st}, we generalize results of Str{\o}mme in \cite{S} to relative Quot schemes over $\pp^1$ bundles. Section \ref{sub} describes the tangent spaces to splitting degeneracy loci along open strata.
In Section \ref{classes},
we find classes of certain splitting loci where the techniques of Example \ref{-1,1} readily generalize. Finally, Section \ref{gen} proves Theorem \ref{main} with an inductive procedure that computes the classes of all splitting degeneracy loci.

\subsection*{Acknowledgements} I would like to thank Ravi Vakil for many helpful conversations and Eric Larson for his assistance implementing code to perform calculations in Example \ref{-202}. Thanks also to David Eisenbud for drawing my attention to \cite[Conj. 5.1]{ES} and subtleties regarding the scheme structure on splitting loci.
 I am grateful to the Hertz Foundation, the NSF Graduate Fellowship, and the Stanford Graduate Fellowship for their generous support.

\section{Splitting degeneracy loci} \label{sch}
Given a vector bundle $E$ on $\pp^1$, knowing the splitting type of $E$ is equivalent to knowing the list of integers $h^0(\pp^1, E(m))$ for $m \in \zz$, or equivalently the list of integers $h^1(\pp^1, E(m))$ for $m \in \zz$. The multiplicity of $\O(-j)$ as a summand of $E$ is equal to the second difference function evaluated at $j$ of the Hilbert function $m \mapsto h^0(\pp^1, E(m))$ (see e.g. \cite[Lemma 5.6]{ES}).

Let $W$ be a rank $2$ vector bundle on a scheme $B$ and let $\pp W := \Proj( \Sym^{\bullet} W^\vee)$ with projection map $\pi : \pp W \rightarrow B$. On $\pp W$ there is a natural surjection $\pi^* W^\vee \rightarrow \O_{\pp W}(1)$, and on $B$ an isomorphism $W^\vee \cong \pi_* \O_{\pp W}(1)$. Note that if $L$ is a line bundle on $B$, then there is a natural isomorphism $\pp (W \otimes L) \cong  \pp W$ via which $\O_{\pp (W \otimes L)}(1) = \O_{\pp W}(1) \otimes \pi^* L^\vee$. By a $\pp^1$ bundle $\pp W$ on $B$, we will mean to remember the data of the rank $2$ vector bundle $W$, or equivalently a choice of relative degree $1$ line bundle $\O_{\pp W}(1)$.

Given a vector bundle $E$ on $\pp W$, we write $E(m)$ for $E \otimes \O_{\pp W}(1)^{\otimes m}$. 
Upper-semicontinuity of the ranks of cohomology of $E(m)$ on fibers determines which splitting loci can be in the closures of others. Given two splitting types $\vec{e} = (e_1, \ldots, e_r)$ with $e_1 \leq \cdots \leq e_r$ and ${\vec{e}} \nobreak\hspace{.16667em plus .08333em}' = (e_1', \ldots, e_r')$ with $e_1' \leq \cdots \leq e_r'$, we define a partial ordering  by $\vec{e}\nobreak\hspace{.16667em plus .08333em}' \leq \vec{e}$ if all partial sums $e_1' +\ldots + e_k' \leq e_1 + \ldots +e_k$. For each rank and degree, there is a unique maximal splitting type called the \textit{balanced splitting type}, characterized by the condition that $|e_i - e_j| \leq 1$.
Set theoretically, \textit{splitting degeneracy loci} are defined by
\[\overline{\Sigma}_{\vec{e}} := \bigcup_{\vec{e}\nobreak\hspace{.16667em plus .08333em}' \leq \vec{e}} \Sigma_{\vec{e}\nobreak\hspace{.16667em plus .08333em}'}.\]
\begin{warning}
The locus $\overline{\Sigma}_{\vec{e}}$ is always closed, but in general, it may not be equal to the closure of $\Sigma_{\vec{e}}$. If the splitting loci for $\vec{e}\nobreak\hspace{.16667em plus .08333em}' < \vec{e}$ have the expected codimension then the support of $\overline{\Sigma}_{\vec{e}}$ is the closure of $\Sigma_{\vec{e}}$ (see \cite[Ch.~14]{EH}). Because splitting loci on the moduli space of vector bundles on $\pp^1$ bundles occur in the correct codimension, the support of $\overline{\Sigma}_{\vec{e}}$ is the pullback of the closure of the universal stratum of splitting type $\vec{e}$, motivating this notation.
\end{warning}

\begin{remark}
Recent work of Geoffrey Smith studies degenerations of splitting types as $\pp^1$ degenerates to a tree of genus zero curves \cite{Geoff}.
\end{remark}

The cohomological conditions determining a splitting type and the theorem on cohomology and base change show that each splitting degeneracy locus $\overline{\Sigma}_{\vec{e}}(E)$ is a finite intersection of loci
\[\{b \in B: h^1(E(m)|_{\pi^{-1}(b)} ) \geq n\} = \{b \in B: \dim (R^1\pi_*E(m))_b \geq n\}.\]
The latter has a natural scheme structure as defined by the $(n-1)$st Fitting ideal of $R^1\pi_*E(m)$.
Let $\Sigmabar_{\vec{e}}(E)$ be intersection of these schemes.
 In Section \ref{sub}, we describe the tangent space to this intersection along the open stratum $\Sigma_{\vec{e}}(E) \subset \overline{\Sigma}_{\vec{e}}(E)$. As discussed there, the geometry of $\Sigmabar_{\vec{e}}(E)$ along the more unbalanced loci is more subtle. 

When $E$ has rank $2$, the possible splitting types are totally ordered with respect to $\leq$. However, in general they need not be.

\begin{Example}[Splitting type $(-2, 0, 2)$, to be revisited in Example \ref{-202}] \label{s2ex}
The diagram below describes splitting loci for a rank $3$ degree $0$ vector bundle $E$ on a $\pp^1$ bundle. The cohomological conditions determining each splitting type are listed below it. An arrow between types indicates when one splitting type is below another in the partial ordering. Recall that the expected codimension for splitting type $\vec{e}$ is $u(\vec{e}) := h^1(\pp^1, End(\O(\vec{e})))$.

\vspace{.1in}
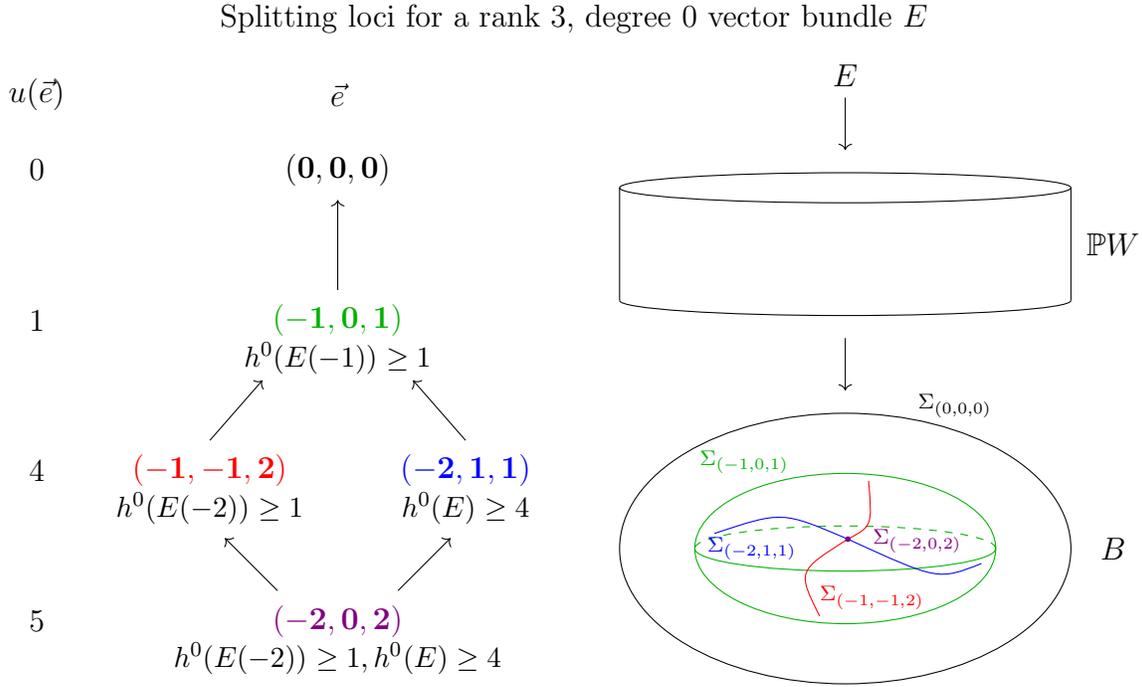
\begin{center}
Splitting loci for a rank $3$, degree $0$ vector bundle $E$

\vspace{.1in}
\begin{tikzpicture}
\draw (0, 11.5) node {$\vec{e}$};
\draw (-4, 11.5) node {$u(\vec{e})$};
\draw (0, 10.5) node {$\mathbf{(0, 0, 0)}$};
\draw [->] (0, 8.9) -- (0, 10.1);
\draw [black!30!green] (0, 8.5) node {$\mathbf{(-1, 0, 1)}$};
\draw [->] (-1.7, 6.9) -- (-1,7.9-.2);
\draw [->] (1.7, 6.9) -- (1,7.9-.2);
\draw (0, 8) node {\small $h^0(E(-1)) \geq 1$};
\draw [red] (-1.7, 6.5) node {$\mathbf{(-1, -1, 2)}$};
\draw (-1.7,6) node {\small $h^0(E(-2)) \geq 1$};
\draw (1.7,6) node {\small $h^0(E) \geq 4$};
\draw [blue] (1.7, 6.5) node {$\mathbf{(-2, 1, 1)}$};
\draw [->] (-.8, 4.9) -- (-1.5, 5.6);
\draw [->] (.8, 4.9) -- (1.5, 5.6);
\draw [violet] (0, 4.5) node {$\mathbf{(-2, 0, 2)}$};
\draw (0, 4) node {\small $h^0(E(-2)) \geq 1, h^0(E) \geq 4$};
\draw (-4, 10.5) node {$0$};
\draw (-4, 8.5) node {$1$};
\draw (-4, 6.5) node {$4$};
\draw (-4, 4.5) node {$5$};
\end{tikzpicture}
\hspace{.3in}
\begin{tikzpicture}
\draw [->] (3-.16, 10.5) node [above] {$E$} -- (-.16+3, 9.5+.3);
 \draw (2.84+3,4.5+3+.3) arc(0:-180:3cm and .2cm);
\draw (2.84,4.5+4.5+.3) ellipse (3cm and .2cm);
\draw (-.16, 4.5+4.5+.3) -- (-.16, 4.5+3+.3);
\draw (-.16+6, 4.5+4.5+.3) -- (-.16+6, 4.5+3+.3);
\draw (6.4, 4.5+3+.75+.3) node {$\pp W$};
\draw [->] (3-.16, 7.5-.2) -- (3-.16, 6.5+.1);
\draw (6.4, 4.5) node {$B$};
\draw [black!30!green] (1.5, 4+.65+1) node {\tiny $\Sigma_{(-1,0,1)}$};
\draw (2.84,4.5) ellipse (3cm and 1.8cm);
\draw [black!30!green] (2.84,4.5) ellipse (2cm and 1cm);
\draw [black!30!green] (2.84+2,4.5) arc (0:-180: 2cm and .3cm);
\draw [dashed, black!30!green] (2.84+2,4.5) arc (0:180: 2cm and .3cm);
\draw [blue] (1.1,4.7) .. controls (2, 4.1+.4+.5) .. (2.5+.25+.125, 4.5+.125);
\draw [blue] (2.5+.25+.125, 4.5+.125) .. controls (2.5+.25+.5+.125 +.7, 4.5+.125-.65+.1) .. (4.65, 4+.3);
\draw [blue] (1.6, 4.5) node {\tiny $\Sigma_{(-2,1,1)}$};
\draw [red] (2.5, 3.6) .. controls (2.2, 4.2) .. (2.5+.25+.125, 4.5+.125);
\draw [red] (2.5+.25+.125, 4.5+.125)   .. controls (3.2, 4.8) .. (3+.15, 5.4);
\draw [red] (3.2, 3.85) node {\tiny $\Sigma_{(-1, -1, 2)}$};
\draw [violet] (2.5+.25+.125, 4.5+.125) node [circle,fill,inner sep=.75pt]{};
\draw [violet] (2.8+1, 4.6) node {\tiny $\Sigma_{(-2, 0, 2)}$};
\draw (4.3, 6.4) node {\tiny $\Sigma_{(0,0,0)}$};
\end{tikzpicture}
\end{center}

\vspace{.05in}
\noindent
Note that the $(-2, 0, 2)$ splitting locus is not determined by a single rank condition. When splitting loci occur in the correct codimension, $\overline{\Sigma}_{(-2, 0, 2)}$ is the intersection of $\overline{\Sigma}_{(-1, -1, 2)}$ and $\overline{\Sigma}_{(-2, 1, 1)}$ but the intersection is not transverse. This makes the task of computing splitting loci a delicate one in general. 
\end{Example}

A similar failure of transversality occurs in \cite{F1, F2} where Fulton studies
simultaneous rank conditions of maps of vector bundles. 
Fulton's solution involves working on a flag bundle over the base and pushing forward a class found there. 
In a similar spirit, the above failure of transversality leads us to work on certain relative Quot schemes for our $\pp^1$ bundle.

\section{Relative Quot schemes of $\pp^1$ bundles} \label{st}
Several of the results in this section generalize Str{\o}mme's work concerning vector bundles on trivial $\pp^1$ bundles \cite{S}  to the case of non-trivial $\pp^1$ bundles. Many of his proofs hold with appropriate modifications.

A key ingredient for explicit computation with splitting loci will be canonical resolutions of $R^1 \pi_*E^\vee(-m)$ for certain $m$. To motivate these resolutions, we first explain the situation on a fixed $\pp^1$.
Suppose $E$ is a globally generated vector bundle of rank $r$ and degree $k$ on $\pp^1$, so there is a canonical surjection $H^0(\pp^1, E) \otimes \O_{\pp^1} \rightarrow E$ and $h^0(\pp^1, E) = \chi(\pp^1, E) = r + k$. It follows that the kernel of this surjection is rank $k$, degree $-k$ and therefore equal to $\O_{\pp^1}(-1)^{\oplus k}$ (as a subbundle of a trival bundle, all summands of the kernel are non-positive; moreover, any trivial summand would give a linear relation among the global sections). We also have $h^0(\pp^1, E(-1)) = k$, so we can summarize the above observations with a sequence
\[0 \rightarrow H^0(\pp^1, E(-1)) \otimes \O_{\pp^1}(-1) \rightarrow H^0(\pp^1, E) \otimes \O_{\pp^1} \rightarrow E \rightarrow 0.\]
The following lemma shows that this sequence globalizes suitably over $\pp^1$ bundles.
This generalizes \cite[Prop. 1.1]{S}, which proves the case of trivial $\pp^1$ bundles.
\begin{Lemma} \label{meS}
Let $E$ be a vector bundle on a $\pp^1$ bundle $\pi: \pp W \rightarrow B$ and let $L = \det W^\vee$ on $B$. 
If $R^1\pi_*E(-1) = 0$ then there is a short exact sequence on $\pp W$
\[0 \rightarrow \pi^*(L \otimes \pi_*E(-1))(-1) \rightarrow \pi^*\pi_* E \rightarrow E \rightarrow 0.\]
\end{Lemma}
\begin{proof}
Str{\o}mme's proof generalizes with suitable care. Let $X$ be the fiber product of $\pp W \rightarrow B$ with itself and consider the diagonal $\Delta \subset X$:
\begin{center}
\begin{tikzcd}
& \Delta\arrow[r,symbol=\subset] &[-25pt] X \arrow{r}{p} \arrow{d}[swap]{q} &\pp W \arrow{d}{\pi} \\
& & \pp W \arrow{r}[swap]{\pi} & B.
\end{tikzcd}
\end{center}
Suppose $W$ is trivialized on some open by sections $e_0$ and $e_1$. Let $x_i = p^* e_i^\vee$ and $y_i = q^* e_i^\vee$ be the pullbacks of the corresponding dual sections of $\O_{\pp W}(1)$. Then $\Delta$ is cut out locally by the vanishing of $(x_0 e_0 + x_1 e_1) \wedge (y_0 e_0 + y_1 e_1) = (x_0 y_1 - x_1 y_0) e_0 \wedge e_1$. This globalizes to realize $\Delta$ as the vanishing of a section of $q^* \pi^* \det W \otimes p^*\O_{\pp W}(1) \otimes q^* \O_{\pp W}(1)$. (In fact, this is the unique functorial construction of a line bundle of the correct degrees on fibers of $p$ and $q$ which is unaffected by twisting $W$ by a line bundle on the base.) In particular, we have an exact sequence
\[0 \rightarrow q^* \pi^* L \otimes p^* \O_{\pp W}(-1) \otimes q^* \O_{\pp W}(-1) \rightarrow \O_X \rightarrow \O_\Delta \rightarrow 0.\]
Following Str{\o}mme, we tensor with $p^* E$, apply $q_*$ and use the projection formula to obtain a long exact sequence on $\pp W$:
\begin{align} \label{les}
0 &\rightarrow \pi^*L \otimes q_* p^* E(-1) \otimes \O_{\pp W}(-1) \rightarrow q_* p^* E \rightarrow E \\
&\rightarrow \pi^*L \otimes R^1 q_* p^* E(-1) \otimes \O_{\pp W}(-1) \rightarrow R^1q_* p^* E \rightarrow 0. \notag
\end{align}
By the theorem on cohomology and base change, $R^1q_* p^* E(-1) = \pi^* R^1\pi_* E(-1) = 0$, so the first row is exact.
Similarly, $q_* p^* E(-1) = \pi_* \pi^* E(-1)$ and $q_* p^* E = \pi^* \pi_* E$, producing the desired sequence.
\end{proof}

We will apply Lemma \ref{meS} to suitable twists of vector bundles $E$.
For an integer $m$, the condition $R^1\pi_* E(m-1) = 0$ is equivalent to the condition that the restriction of $E(m)$ to each fiber is globally generated, which in turn is equivalent to saying all summands of $E(m)$ restricted to any fiber are non-negative. In this case, taking the dual of the sequence in Lemma \ref{meS} expresses $E^\vee(-m)$ as the kernel of a map between twists of pullbacks of vector bundles from the base:
\begin{equation} \label{onP}
0 \rightarrow E^\vee(-m) \rightarrow \pi^*(\pi_*E(m))^\vee \xrightarrow{\psi} \pi^* (L \otimes \pi_*E(m-1))^\vee(1) \rightarrow 0.
\end{equation}
Pushing forward, and recalling that $\pi_* \O_{\pp W}(1) \cong W^\vee$, we obtain
\begin{equation} \label{onB}
0 \rightarrow \pi_*E^\vee(-m) \rightarrow (\pi_* E(m))^\vee \xrightarrow{\pi_* \psi} \pi^* (L \otimes \pi_* E(m-1))^\vee \otimes W^\vee \rightarrow R^1 \pi_* E^\vee(-m) \rightarrow 0.
\end{equation}
Sections \ref{classes} and \ref{gen} take advantage of these sequences to compute classes of splitting degeneracy loci.

The proof of Lemma \ref{meS} also relates push forwards of various twists of $E$ in the $K$-theory of $B$. Let $R\pi_* E$ denote the derived push forward $[\pi_* E] - [R^1\pi_* E]$ in $K(B)$. In addition, we define the following class in $K$ theory depending only on $W$
\[\Theta(m) := \sum_{i = 0}^{\lfloor m/2 \rfloor}(-1)^i {m-i \choose i} W^{\vee \otimes m-2i} \otimes L^{\vee \otimes m-i} \in K(B).\]
Note that $\rank \Theta(m) = m+1$, where $\rank$ is understood to extend linearly to $K$-theory.
\begin{Lemma} \label{ind}
If $E$ is a vector bundle on $\pi: \pp W \rightarrow B$, then
\[R\pi_* E(-1) = R\pi_*E \otimes W^\vee \otimes L^\vee - R\pi_* E(1) \otimes L^\vee\]
in $K(B)$. More generally, by induction it follows that
\[R\pi_*E(-1) = (\Theta(m+1) - \Theta(m) \otimes W^\vee \otimes L^\vee) \otimes R\pi_* E(m) + \Theta(m) \otimes R\pi_* E(m-1).\]
\end{Lemma}
\begin{remark}
The advantage of the second expression is that for suitably large $m$, $R\pi_* E(m) = \pi_*E(m)$ and $R\pi_* E(m-1) = \pi_*E(m-1)$ are vector bundles on $B$.
\end{remark}
\begin{proof}
The first statement follows from tensoring \eqref{les} by $\O_{\pp W}(1)$ and pushing forward to $B$. 
Setting $b_i = R\pi_* E(m-i)$, $x = W^\vee \otimes L^\vee,$ and $y = -L^\vee$, we obtain a two-term recurrence relation of the form $b_{m+1} = xb_{m} + yb_{m-1}$ for all $m$.
Packaging these in a generating function $f(t) = \sum_{i\geq 0} b_i t^i$, we see 
\[f(t) = \frac{b_0 + b_1 t - xt b_0}{1-xt - yt^2} \quad \Rightarrow \quad b_{m+1} = \theta(m) b_0 + \theta(m-1)(b_1 - xtb_0),\]
where $\theta(m)$ denotes the coefficient of degree $m$ in the rational function $\frac{1}{1-xt - yt^2}$. Solving for this coefficient recovers our definition of $\Theta(m)$.
\end{proof}

In \cite{S}, Str{\o}mme describes an embedding of the Quot scheme of a trivial vector bundle on $\pp^1$ into a product of Grassmannians. We require a generalization to relative Quot schemes over $\pp^1$ bundles. Let $\mathcal{F}$ be a vector bundle on $U$ and $\pi: \pp W \rightarrow U$ a $\pp^1$ bundle. Given some Hilbert polynomial,
 let $\Quot_{\pi^*\F}$ denote the relative Quot scheme of $\pi^* \F$ over $\pp W \rightarrow U$. The scheme $\Quot_{\pi^*\F}$ can be thought of as a fiber bundle over $U$ where the fiber over $b \in U$ is Str{\o}mme's corresponding Quot scheme of the trivial bundle $\F_b \otimes \O_{\pp W_b}$ on $\pp W_b$. Let us label maps in the following commutative diagram:
\begin{equation} \label{di}
\begin{tikzcd}
&\Quot_{\pi^*\F} \times_{U} \pp W \arrow{d}[swap]{p} \arrow{r}{q} &\pp W \arrow{d}{\pi} \\
&\Quot_{\pi^*\F} \arrow{r}[swap]{\gamma} &U.
\end{tikzcd}
\end{equation}
Then $\Quot_{\pi^*\F} \times_U \pp W$ is equipped with a tautological sequence
\begin{equation} \label{ts}
0 \rightarrow \mathcal{S} \rightarrow q^*\pi^* \mathcal{F} \rightarrow \mathcal{Q} \rightarrow 0,
\end{equation}
where $\mathcal{Q}$ is flat over $\Quot_{\pi^*\F}$. 
Let $-d$ be the degree of $\mathcal{S}$ restricted to a fiber of $p$. For each $m \geq d-1$, tensoring \eqref{ts} with $q^*\O_{\pp W}(m)$ and pushing forward to $\Quot_{\pi^*\F}$ gives rise to a natural injection
\begin{equation} \label{maps}
p_* \mathcal{S}(m) \hookrightarrow p_* (q^*(\pi^*\mathcal{F})(m)) = \gamma^*(\F \otimes \Sym^{m} W^\vee).
\end{equation}
This induces a map of $\Quot_{\pi^*\F}$ to a corresponding Grassmann bundle over $U$. The following generalizes \cite[Thm. 4.1]{S}.

\begin{theorem} \label{em}
Let $\Quot_{\pi^*\F}$ be the relative Quot scheme of $\pi^*\mathcal{F}$ over $\pi: \pp W \rightarrow U$ for some Hilbert polynomial and let $p$ be as in \eqref{di}. Let $-d$ be the relative degree of the tautological subbundle $\mathcal{S}$ and let $r_{d-1} = \rank p_* \S(d-1)$ and $r_d = \rank p_* \S(d)$. There is an embedding
\begin{center}
\begin{tikzcd}
&\Quot_{\pi^*\F} \arrow{r}{\iota} \arrow{dr}[swap]{\gamma} &G(r_{d-1}, \F \otimes Sym^{d-1}W^\vee) \times G(r_d, \F \otimes Sym^d W^\vee) \arrow{d}{\rho} \\
& & U
\end{tikzcd}
\end{center}
 such that the tautological subbundles $S_{d-1}$ and $S_d$ on the Grassmann bundles on the right restrict to $p_*\mathcal{S}(d-1)$ and $p_*\mathcal{S}(d)$. Moreover, the image of $\Quot_{\pi^*\F}$ has class 
 \[[\Quot_{\pi^*\F}] = c_{top}(S_{d-1}^\vee \otimes Q_d \otimes \rho^*(W^\vee \otimes L^\vee)),\]
where $Q_d$ denotes the tautological quotient bundle on the second factor Grassmann bundle.
\end{theorem}
\begin{proof}
To see the map is an embedding, it suffices to check on fibers of $\Quot_{\pi^*\F} \rightarrow U$, which reduces us to Str{\o}mme's setting.
To determine the image, Str{\o}mme uses a relationship between the maps in \eqref{maps} for adjacent twists.
Tensoring the sequence in Lemma \ref{meS} by $\O_{\pp W}(1)$ and pushing forward gives rise to a natural map $\pi_* E(-1) \rightarrow \pi_*E \otimes W^\vee \otimes L^\vee$. In the case that $E = \pi^* \F \otimes \O_{\pp W}(m)$ for some vector bundle $\F$ on $U$ and $m \geq 1$, this gives a map 
\[\mathcal{F} \otimes \Sym^{m-1}W^\vee  \rightarrow \mathcal{F} \otimes \Sym^{m} W^\vee \otimes W^\vee \otimes L^\vee.\]
The only modification needed in Str{\o}mme's proof is that his natural map $j_m$ on page 262 should be replaced with the above.
This results in replacing Str{\o}mme's $2$-dimensional vector space $H$ by the rank $2$ vector bundle $\rho^*(W^\vee \otimes L^\vee)$ throughout the remainder of his Section 4. His proof then shows $\Quot_{\pi^*\F}$ is the zero locus of a natural map $S_{d-1} \rightarrow Q_d \otimes \rho^*(W^\vee \otimes L^\vee)$ on the product of Grassmann bundles, proving the formula for its class.
\end{proof}

\section{The tangent space to splitting loci} \label{sub}
In this section, we describe the tangent spaces to splitting degeneracy schemes and show they satisfy a certain minimality property.
We also provide an alternative description of the tangent space that will appear in Section \ref{gen}.
Recall that, as a scheme, we have defined
\[\Sigmabar_{\vec{e}}(E)=\bigcap_{m}  \{b \in B: \dim (R^1\pi_*E(m))_b \geq h^1(\O(\vec{e})(m))\},\]
where the schemes in the intersection on the right are defined by the appropriate Fitting ideals of $R^1\pi_*E(m)$. 

Let $T = \spec k[\epsilon]/(\epsilon^2)$. For $b \in B$, let $\Mor_b(T, B)$ denote the space of morphisms $T \rightarrow B$ sending the reduced point $0 =\spec k \subset T$ to $b$. Given a vector bundle $\E$ on $\pp^1 \times T$, we write $\E_0$ for the restriction to $\pp^1 \times 0$.
 Given any $v: T \rightarrow B$, we have a fibered diagram
 \begin{center}
 \begin{tikzcd}
 {v'}^*E \arrow{r} \arrow{d} & E \arrow{d} \\
 T \times \pp^1 \arrow{d}[swap]{\pi'} \arrow{r}{v'} & \pp W \arrow{d}{\pi} \\
 T \arrow{r}[swap]{v} & B.
 \end{tikzcd}
 \end{center} 
 There is a natural map on tangent spaces
 \[\delta_{E,b}: T_bB = \Mor_b(T, B) \rightarrow \Def(E|_{\pi^{-1}(b)}) = H^1(End(E|_{\pi^{-1}(b)}))\] 
  that sends a map $v : T \rightarrow B$ to the induced first order deformation ${v'}^*E$. The tangent space to any scheme structure on a splitting locus of $E$ contains $\ker (\delta_{E,b})$.
  We demonstrate that our schemes satisfy the following minimality property.
  
  \begin{Lemma} \label{dlem}
For $b \in \Sigma_{\vec{e}}(E) \subset \Sigmabar_{\vec{e}}(E)$, the tangent space is
\[
  T_b\Sigmabar_{\vec{e}}(E) = \ker (\delta_{E,b}). 
\]
  \end{Lemma}
 \begin{remark}
 In \cite[Conj. 5.1]{ES}, Eisenbud-Shreyer conjecture that  $\Sigmabar_{\vec{e}}(E)$ is reduced in the case where $B$ is a versal deformation space of $\O^{\oplus r-1} \oplus \O(d)$.  In this universal setting, $\delta_{E, b}$ is surjective for all $b \in \Sigma_{\vec{e}}(E)$, so Lemma \ref{dlem} shows that $\Sigmabar_{\vec{e}}(E)$ is smooth along $\Sigma_{\vec{e}}(E)$. However, this does not rule out the possibility of embedded points along the more unbalanced locus $\Sigmabar_{\vec{e}}(E) \backslash \Sigma_{\vec{e}}(E)$,
 so \cite[Conj. 5.1]{ES} remains an open conjecture. For the purposes of computing classes of degeneracy loci, our assumption that the more unbalanced locus $\Sigmabar_{\vec{e}}(E) \backslash \Sigma_{\vec{e}}(E)$ occurs in higher codimension means this subtlety does not affect the class.
 \end{remark}
 
First let us identify the tangent space to $\Sigmabar_{\vec{e}}(E)$.

\begin{Lemma}
Let $b \in \Sigma_{\e}(E) \subset \Sigmabar_{\e}(E)$. 
The tangent space to $\Sigmabar_{\vec{e}}(E)$ is
\[T_b\Sigmabar_{\vec{e}}(E) = \{v \in \Mor_b(T, B): R^1\pi'_*({v'}^*E)(m) \text{ is free of rank } h^1(\O(\e)(m)) \ \forall \ m\}.\]
\end{Lemma}
\begin{proof}
Let $F \rightarrow G \rightarrow R^1\pi_*E(m)$ be a locally free resolution on $B$. If $v: T \rightarrow \Sigmabar_{\vec{e}}(E)$, then $v^*F \rightarrow v^*G \rightarrow v^*R^1\pi_*E(m)$ is a free resolution and the appropriate minors of $v^*F \rightarrow v^*G$ vanish on all of $T$. Some minor one size smaller is nonzero at the reduced point, hence a unit. Thus, the cokernel $v^*R^1\pi_*E(m)$ is free of the correct rank. Cohomology and base change shows that $v^*R^1\pi_*E(m) = R^1\pi'_*{v'}^*E(m)$.
\end{proof}

Lemma \ref{dlem} is now implied by the following.
\begin{Lemma}
Let $E$ be a vector bundle on $\pp^1 \times T$ with $E_0 \cong \O(\vec{e})$. Label the projections
\begin{center}
\begin{tikzcd}
& \arrow{dl}[swap]{\alpha} \pp^1 \times T \arrow{dr}{\pi}& \\
\pp^1 & & T
\end{tikzcd}
\end{center}
 Suppose furthermore that $R^1\pi_*E(-m)$ is locally free of rank $h^1(\O(\vec{e})(-m))$ for all $m \geq \min\{e_i\}$. Then $E \cong \alpha^*\O(\vec{e})$ is the trivial deformation.
\end{Lemma}
\begin{proof}
We induct on the rank of $E$. If $E$ is balanced there is nothing to prove, as every deformation is trivial. After twisting, we may write $E_0 \cong \O(\vec{e}) = \O(-1)^{\oplus i} \oplus \O(\vec{a})$ where every summand of $\O(\vec{a})$ is nonnegative. As all summands of $E_0$ have degree $> -2$, cohomology and base change shows that $(R^1\pi_*E)_0 = 0$ and hence $ R^1\pi_*E = 0$.
By hypothesis, $R^1\pi_*E(-1)$ is free of rank $i$.
By the proof of Lemma \ref{meS}, there is a long exact sequence on $T$, 
\[0 \rightarrow \pi^*(\pi_*E(-1))(-1) \rightarrow \pi^*\pi_*E \rightarrow E \rightarrow \pi^*(R^1\pi_*E(-1))(-1) \rightarrow \pi^*R^1\pi_*E \rightarrow 0.\]
In particular, we have a surjection $E \rightarrow \alpha^*\O(-1)^{\oplus i}$. Let $F$ denote the kernel, which is locally free of lower rank with $F_0 \cong \O(\vec{a})$. For each $m \geq 0$, we have an exact sequence on $\pp^1 \times T$
\[0 \rightarrow F(-m) \rightarrow E(-m) \rightarrow \alpha^*\O(-m-1) \rightarrow 0,\]
which pushes forward to give a sequence of vector bundles on $T$:
\[0 \rightarrow R^1\pi_*F(-m) \rightarrow R^1\pi_*E(-m) \rightarrow R^1\pi_*\alpha^*\O(-m-1)^{\oplus i} \rightarrow 0.\]
By hypothesis, $R^1\pi_*E(-m)$ is free of rank $h^1(\O(\vec{e})(-m))$. Meanwhile, the last term $R^1\pi_*\alpha^*\O(-m-1)^{\oplus i}$ is free of rank $h^1(\O(-1)^{\oplus i}(-m))$.
It follows that $R^1\pi_*F(-m)$ is free of rank $h^1(\O(\vec{e})(-m)) - h^1(\O(-1)^{\oplus i}(-m)) = h^1(\O(\vec{a})(-m))$. By induction, $F \cong \alpha^*\O(\vec{a})$ is the trivial deformation. Now we see
\begin{equation} \label{ha}
0 \rightarrow \alpha^*\O(\vec{a}) \rightarrow E \rightarrow \alpha^*\O(-1)^{\oplus i} \rightarrow 0.
\end{equation}
Finally, we have
\[H^1(T, Hom(\alpha^*\O(-1)^{\oplus i}, \alpha^*\O(\vec{a}))) = \bigoplus_{j=1}^{r-i} H^1(T, \alpha^*\O(a_j - 1))^{\oplus i} = 0,\]
so \eqref{ha} must split.
\end{proof}

We also require another description of the tangent space.
\begin{Lemma} \label{other}
Suppose $b \in \Sigma_{\vec{e}}(E)$. Write $\O(\vec{e}) = \O(-m_0)^{\oplus i} \oplus \alpha^*\O(\vec{a})$ where $a_i > -m_0$ for all $i$. Then
\[ \ker(\delta_{E,b}) =\{v \in \Mor_b(T, B) : \alpha^*\O(\vec{a}) \text{ is a subsheaf of } v^*E\} \]
\end{Lemma}
\begin{proof}
The left hand side is automatically contained in the right hand side. After an overall twist, we may assume that $m_0 = 0$. Now suppose $\alpha^*\O(\vec{a})$ is a subsheaf of $v^*E$. Let $Q$ be the quotient, so we have a short exact sequence over $T \times \pp^1$:
\begin{equation} \label{ha2}
0 \rightarrow \alpha^*\O(\vec{a}) \rightarrow E \rightarrow Q \rightarrow 0.
\end{equation}
The $\O(\vec{a})$ subsheaf of $E_0 = \O(\vec{e})$ is unique and is a subbundle. It follows that $Q$ is locally free and $Q_0 \cong \O^{\oplus i}$. Hence $Q \cong \alpha^*\O^{\oplus i}$ is trivial.
Now, because all summands of $\O(\vec{a})$ are positive, $H^1(T, Hom(\alpha^*\O^{\oplus i}, \alpha^*\O(\vec{a}))) = 0$, showing that \eqref{ha2} splits.
\end{proof}
\section{Certain degeneracy classes} \label{classes}

Suppose $E$ is a degree $\ell$, rank $r$ vector bundle on $\pp W \rightarrow B$. As before, let $L = \det W^\vee$. Finding the splitting loci of $E$ is the same problem as finding splitting loci of twists $E(i)$, so from now on, we assume $0 \leq \ell < r$.
We start by computing the classes of degeneracy loci of the form $\overline{\Sigma}_{(-m, *, \ldots, *)}$, where $*$'s indicate a balanced remainder, i.e.
\[(-m, *, \ldots, *) = \left(-m, \left\lfloor \frac{\ell + m}{r-1} \right\rfloor, \ldots,  \left\lceil \frac{\ell+m}{r-1} \right\rceil \right).\]
The expected codimension of $(-m, *, \ldots, *)$ is
\[h^1(\pp^1, End(\O(-m, *, \ldots, *))) = h^1(\pp^1, Bal^\vee(-m)) = r(m-1) + \ell + 1, \]
where $Bal$ denotes the balanced bundle of rank $r-1$ and degree $\ell + m$.

Assuming $\overline{\Sigma}_{(-m-1, *, \ldots, *)}$ occurs in higher codimension than $\overline{\Sigma}_{(-m, *, \ldots, *)}$, excision (see e.g. \cite[Prop. 1.14]{EH}) allows us to calculate the class of $\overline{\Sigma}_{(-m, *, \ldots, *)}$ on the open set $U_m = B \backslash \overline{\Sigma}_{(-m-1, *, \ldots, *)}$. By the theorem on cohomology and base change, over $U_m$, the pushforwards $\pi_*E(m)$ and $\pi_*E(m-1)$ are locally free of rank $(m+1)r + \ell$ and $mr+\ell$ respectively. Let $\mathcal{F} := (\pi_*E(m))^\vee$ and $\mathcal{G} := (L \otimes \pi_*E(m-1))^\vee$. Now equation \eqref{onB} becomes
\[0 \rightarrow \pi_*E^\vee(-m) \rightarrow \mathcal{F} \xrightarrow{\pi_*\psi} \mathcal{G}\otimes W^\vee \rightarrow R^1\pi_* E^\vee(-m) \rightarrow 0.\]
We have that $\overline{\Sigma}_{(-m, *, \ldots, *)}$ is precisely the locus where $\pi_*\psi: \F \rightarrow \G \otimes W^\vee$ fails to be injective on fibers. 
The expected codimension of this locus as a degeneracy locus of a map of vector bundles is
\[\rank (\G \otimes W^\vee) - \rank \F + 1  = r(m-1) + \ell + 1.\]
Therefore, applying Porteous' formula (see e.g. \cite[Thm. 12.4]{EH}) proves the following.
\begin{Lemma}
If $\codim \overline{\Sigma}_{(-m, *, \ldots, *)} = r(m-1)+\ell+1$ and $\codim \overline{\Sigma}_{(-m-1, *, \ldots, *)} > r(m-1) + \ell + 1$, then
\begin{equation} \label{umc}
[\overline{\Sigma}_{(-m,*,\ldots,*)}]=\left[\frac{c(\G \otimes W^\vee)}{c(\F)} \right]_{r(m-1)+\ell+1}
\end{equation}
where we formally invert the denominator and the subscript indicates that we take the component of the resulting class in that degree.
\end{Lemma}

\begin{remark}
One might hope more generally to compute the class of $\overline{\Sigma}_{(-m^i, *, \ldots, *)}$ as the locus where $\dim \ker \psi \geq i$. The codimension is correct to apply Porteous' formula, and this does indeed yield the class of $\overline{\Sigma}_{(-m^i,*, \ldots,*)}$ \textit{restricted to} $U_m$. However, in general, $\codim U_m^c$ may be smaller than $\codim \overline{\Sigma}_{(-m^i,*, \ldots,*)}$, so the class can contain contributions from $U_m^c$. Nevertheless, this is approach is helpful if one has a family that is bounded in some way (see for example \cite[Section 4]{L}).
\end{remark}

\begin{remark}
In the case $U_m$ is irreducible and projective, 
Fulton-Lazarsfeld's theorem on connectedness of degeneracy loci \cite{FL} shows that $\overline{\Sigma}_{(-m^i,*, \ldots,*)}$ is connected if $\F^\vee \otimes \G \otimes W^\vee$ is ample.
\end{remark}

\section{The inductive algorithm} \label{gen}

In general, splitting degeneracy loci are determined by a sequence of cohomological conditions on the fibers, but the conditions are not transverse (see Example \ref{s2ex}).
This indicates that we need something more refined than Porteous calculations on the base.

The Porteous formula finds the class of where a map of vector bundles drops rank by pulling back to a Grassmann bundle, computing the class of where the universal subbundle includes into the kernel and pushing forward the result. 
To get an algorithm for arbitrary splitting types, instead of tracking degeneracy of the map $\pi_*\psi$ in \eqref{onB}, we need to track high degree subsheaves of the kernel of the map $\psi$ in \eqref{onP}. We do this by pulling back to an appropriate relative Quot scheme.
Utilizing Theorem \ref{em}, our answer also winds up being a pushforward of natural classes on a product of Grassmann bundles (and reduces to the Porteous formula in the special case). 

\subsection*{Algorithm and Proof of Theorem \ref{main}}
Fix some splitting type $\vec{e}$. We assume that $\Sigma_{\vec{e}}$ is codimension $u(\vec{e})$ and $Y = \overline{\Sigma}_{\vec{e}} \backslash \Sigma_{\vec{e}}$ has codimension greater than $u(\vec{e})$.
Inductively, we can assume we know a formula for the expected classes of splitting degeneracy loci for lower rank bundles, in terms of Chern classes of pushforwards of twists of the vector bundle.
As before, let $U_m = B \backslash \overline{\Sigma}_{(-m-1,*, \ldots, *)}$. 
Fix $m$ large enough that $(\O(\vec{e}))(m)$ is globally generated and $\codim U_m^c > h^1(End(\O(\vec{e}))$.
We will carry out our calculation of the class of $\overline{\Sigma}_{\vec{e}}$ on $U = U_m \backslash Y$, allowing us to assume $\overline{\Sigma}_{\vec{e}} = \Sigma_{\vec{e}}$. The result will hold on all of $B$ by excision.

Let $\F := (\pi_*E(m))^\vee$ and $\G := (L\otimes \pi_*E(m-1))^\vee$ be the vector bundles on $U$ as in the previous section.
Then \eqref{onP} becomes the exact sequence
\begin{equation} \label{forcod}
0 \rightarrow E^\vee(-m) \xrightarrow{\psi} \pi^* \F \rightarrow (\pi^* \G)(1) \rightarrow 0 
\end{equation}
on $\pp W$.
Finding where $E$ has splitting type $\O(\vec{e})$ is the same finding where $E^\vee(-m)$ has splitting type $\O(\vec{e})^\vee(-m)$. Let us write $\O(\vec{e})^\vee(-m) =\O(-m_0)^{\oplus i} \oplus \O(\vec{a})$ where each $a_j > -m_0$. Let $d =- \deg \O(\vec{a})$ and $s = \rank \O(\vec{a}) = r - i$. Any vector bundle admitting a subsheaf of splitting type $\vec{a}$ is at least as unbalanced as $\O(\vec{e})^\vee (-m)$. Therefore, our splitting locus is also described as 
\[\overline{\Sigma}_{\vec{e}} = \{b:  \text{ there exists } \O(\vec{a}') \hookrightarrow \ker \psi_b \text{ for $\vec{a}' \leq \vec{a}$}\}.\]

To describe the latter, let $\Quot_{\pi^*\F}$ be the relative Quot scheme of $\pi^* \F$ over $\pp W \rightarrow U $ parametrizing quotients with Hilbert polynomial $h(n) = (\rank \F - s)(n+1) + d$. On a curve, a subsheaf of a locally free sheaf is locally free, so this is equivalent to parametrizing locally free subsheaves of rank $s$ and degree $-d$.
Thus, we think of $\Quot_{\pi^*\F}$ as a fiber bundle over $U$ where the fiber over $b \in U$ is the Quot scheme parametrizing all quotients of $\F_b \otimes \O_{\pp W_b}$ where the subsheaf has rank and degree equal to $\O(\vec{a})$. Let us label maps of the fiber product as in \eqref{di} and the tautological bundles as in \eqref{ts}.

Consider the composition
\[\phi: \mathcal{S} \rightarrow q^*\pi^*\F \xrightarrow{q^*\psi} q^* (\pi^*\G)(1).\]
We have $\mathcal{S}_b \hookrightarrow (\ker \psi)_b$ when $\phi_b$ is the zero map. In other words, when $\phi$ vanishes, considered as a section of the vector bundle 
\[p_*Hom(\mathcal{S},q^*(\pi^*\G)(1)) = p_*(\mathcal{S}^\vee \otimes \O(1) \otimes p^*\gamma^*\G) = p_*(\mathcal{S}^\vee(1)) \otimes \gamma^*\G,\]
which is locally free by the theorem on cohomology and base change.
Let $\sigma$ denote the top Chern class of this vector bundle and let $Z_{\vec{a}}$ be the closure of the splitting locus in $\Quot_{\pi^*\F}$ over which $\mathcal{S}$ splits as $\O(\vec{a})$.  The splitting loci of $\mathcal{S}$ on each fiber of $\Quot_{\pi^*\F} \rightarrow U$ are the splitting loci on Str{\o}mme's Quot scheme. These are all described as quotients of open subsets of $Hom(\O(\vec{a}'), \F_b \otimes \O_{\pp^1})$ by $Aut(\O(\vec{a}'))$ and hence occur in the expected codimension.

\begin{Lemma} \label{zing}
We have $[\overline{\Sigma}_{\vec{e}}] = \gamma_*(\sigma \cdot [Z_{\vec{a}}]).$
\end{Lemma}
\begin{proof}
With our assumption $\overline{\Sigma}_{\vec{e}} = \Sigma_{\vec{e}}$, any subsheaf of $\ker \psi_b$ of splitting type $\vec{a}'\leq \vec{a}$ is unique (and actually $\vec{a}' = \vec{a}$).
Thus, by construction,
$\gamma$ sends $V(\phi) \cap Z_{\vec{a}}$ one-to-one onto $\overline{\Sigma}_{\vec{e}}$. Lemmas \ref{dlem} and \ref{other} show that $\gamma$ is an isomorphism on tangent spaces.
 In particular, $[\overline{\Sigma}_{\vec{e}}] = \gamma_*([V(\phi) \cap Z_{\vec{a}}])$. 
It also follows that
\begin{align*}
\codim V(\phi) \cap Z_{\vec{a}} &= u(\vec{e}) + \mathrm{fiber dim}(\gamma) \\
&= \codim Z_{\vec{a}} + h^1(\pp^1, End(\O(-m_0)^i \oplus Bal))  + \mathrm{fiber dim}(\gamma),
\end{align*}
where $Bal$ denotes the balanced bundle of rank $s$ and degree $-d$.
The fibers of $\gamma$ have dimension $h^0(\pp^1, Hom(Bal, \O^{\oplus \rank \F}) - h^0(\pp^1, End(Bal))$.  Therefore, the codimension of $V(\phi)$ inside $Z_{\vec{a}}$ is
\begin{equation}\label{codim}
h^1(\pp^1, Hom(Bal,\O(-m_0)^{\oplus i})) + h^0(\pp^1, Hom(Bal, \O^{\oplus \rank \F})) - h^0(\pp^1, End(Bal)).
\end{equation}
Now apply $Hom(Bal, -)$ to the exact sequence
\[0 \rightarrow \O(-m_0)^{\oplus i} \oplus Bal \rightarrow \O^{\oplus \rank \F} \rightarrow \O(1)^{\oplus \rank \G} \rightarrow 0\]
on $\pp^1$ and use the long exact sequence in cohomology to see that \eqref{codim} is equal to $h^0(\pp^1, Hom(Bal, \O(1)^{\oplus \rank \G})) = \rank(p_*(\mathcal{S}^\vee(1)) \otimes \gamma^*\G)$. This shows that $Z_{\vec{a}}$ and $V(\phi)$ meet in the expected codimension. Let us write $V(\phi) = V_0 \cup V_1$ where $V_0$ is the expected codimension and every component of $V_1$ has strictly larger dimension.
Then $\sigma$ differs from $[V_0]$ by a class supported in $V_1$, and $V_1 \cap Z_{\vec{a}} = \varnothing$. Therefore,
\[\sigma \cdot [Z_{\vec{a}}] = [V_0] \cdot [Z_{\vec{a}}] = [V_0 \cap Z_{\vec{a}}] = [V(\phi) \cap Z_{\vec{a}}]\]
and the result follows.
\end{proof}


 To compute the pushforward in Lemma \ref{zing}, we use Theorem \ref{em} to embed $\Quot_{\pi^*\F}$ into a product of Grassmann bundles and adopt the notation of that diagram.
 Our goal is to express $\sigma$ and $[Z_{\vec{a}}]$ as pullbacks of natural classes under $\iota^*$.
First, consider the Chern classes of $(p_*\mathcal{S}^\vee(1))^\vee$. 
Using Serre duality and Lemma \ref{ind}, we have the following equality in $K$-theory:
 \begin{align}
 (p_*\mathcal{S}^\vee(1))^\vee &= R^1p_*(\mathcal{S}(-1) \otimes q^*\O_{\pp W}(-2) \otimes p^*\gamma^*L) = -Rp_* \S(-3) \otimes \gamma^*L \notag \\
 &=(\Theta(d+2) \otimes W^\vee - \Theta(d+3) \otimes L) \otimes p_* \S(d) - \Theta(d+2) \otimes p_* \S(d-1) \otimes L. \label{ps}
 \end{align}
In the second line, we have used $Rp_* \S(d-1) = p_* \S(d-1)$ and the pullbacks by $\gamma$ are implicit. 
This determines a polynomial $\beta_{d}$ in the Chern classes of $W, S_{d-1}, S_d$ such that $c((p_*\mathcal{S}^\vee(1))^\vee) = \iota^*\beta_d$.
For example, in the case where $W$ is trivial, \eqref{ps} simplifies to
\[  (p_*\mathcal{S}^\vee(1))^\vee = (d+2)p_* \S(d) - (d+3)p_*\S(d-1) \qquad \text{from which} \qquad \beta_d = \frac{c(S_d)^{d+2}}{c(S_{d-1})^{d+3}}, \]
where we formally invert the denominator.
To obtain an expression for $\sigma$, we use a formula for the top Chern class of a tensor product of vector bundles (see e.g. \cite[Cor. 12.3]{EH}). Noting that $\rank (p_*\mathcal{S}^\vee(1))^\vee = d + 2s$, this gives
\begin{equation} \label{beta}
\sigma = c_{top}( p_*(\mathcal{S}^\vee(1)) \otimes \gamma^*\G)= \Delta^{d+2s}_{\rank \G}\left(\frac{c(\gamma^*\G)}{c((p_*(\mathcal{S}^\vee(1)))^\vee)} \right) = \iota^*\Delta^{d+2s}_{\rank \G} \left(\frac{c(\rho^*\G)}{\beta_d}\right).
 \end{equation}
Above, $\Delta^a_b$ denotes the standard determinantal class: given an input class $x$ in Chow, let $x_i$ be the component in degree $i$ and define
\[\Delta^a_b(x) = \Delta_{b,\ldots,b}(x) = \det\left(\begin{matrix}x_b & x_{b+1} & \cdots & x_{b+a-1} \\ x_{b-1} & x_{b} & \cdots & x_{b+a-2} \\ \vdots & \vdots & \ddots & \vdots  \\ x_{b-a+1} & x_{b-a+2} & \cdots & x_{b}\end{matrix} \right).\]

 Since $\rank \mathcal{S} < r$,
by our inductive hypothesis, we can assume we know a formula for $[Z_{\vec{a}}]$ in terms of Chern classes of bundles $p_*\mathcal{S}(i)$. Using Lemma \ref{ind}, we can write $[Z_{\vec{a}}] = \iota^*\alpha$ where $\alpha$ is some polynomial in the Chern classes of $S_{d-1},S_{d}$ and $W$.
Finally, using push-pull and the class of $\Quot_{\pi^*\F}$ given in Theorem \ref{em}, we have
 \begin{align}
[\overline{\Sigma}_{\vec{e}}] &= \gamma_* (\sigma\cdot [Z_{\vec{a}}] ) = \rho_* \iota_*(\sigma \cdot [Z_{\vec{a}}]) \notag \\
&= \rho_*\iota_*\left(\iota^*\Delta^{d+2s}_{\rank \G} \left(\frac{c(\rho^*\G)}{\beta_d}\right) \cdot \iota^* \alpha\right) \notag \\
 &= \rho_*\left([\Quot_{\pi^*\F}]  \cdot  \Delta^{d+2s}_{\rank \G} \left(\frac{c(g^*\G)}{\beta_d}\right)\cdot \alpha\right) \notag \\
 &= \rho_*\left(c_{top}(S_{d-1}^\vee \otimes Q_d \otimes \rho^*(W^\vee \otimes L^\vee)) \cdot \Delta^{d+2s}_{\rank \G} \left(\frac{c(\rho^*\G)}{\beta_d}\right) \cdot \alpha\right). \label{topush}
 \end{align}
The expression we need to push forward in \eqref{topush} can be solved for explicitly in terms of the pullbacks of Chern classes of $W, \F,$ and $\G$ and the Chern classes of the tautological bundles. The push forwards of all polynomials in the Chern classes of $S_d$ and $S_{d-1}$ are polynomials in the Chern classes of $\F$ and $W$, determined by \cite[Cor. 2.6]{HT}. 
Thus, we have all the necessary ingredients to compute the classes of splitting loci in terms of the Chern classes of $W$, $\mathcal{F} = (\pi_*E(m))^\vee$ and $\mathcal{G} = (L \otimes \pi_*E(m-1))^\vee$.

\begin{Example} [Splitting type $(-2,0,2)$, Example \ref{s2ex} revisited] \label{-202}
We explain how to find the class of $\overline{\Sigma}_{(-2,0,2)}$ using the general algorithm, supposing $W$ is trivial for simplicity.
The stratum $\overline{\Sigma}_{(-2,0,2)}$ is codimension $5$ so we may take $m = 2$. We have 
\[\O((-2,0,2))^\vee(- 2) = \O(-4) \oplus \O(-2) \oplus \O,\]
so $\vec{a} = (-2, 0)$. On $U$, the bundle $\mathcal{F} = (\pi_*E(2))^\vee$ has rank $9$ and $\G = (\pi_*E(1))^\vee$ has rank $6$. We form the relative Quot scheme $\Quot_{\pi^*\F},$ parametrizing rank $2$, degree $-2$ subsheaves of $\F$ on the fibers of $\pp^1 \times U \rightarrow U$. The bundle $p_*\mathcal{S}(1)$ has rank $2$ and $p_*\mathcal{S}(2)$ has rank $4$, so
Theorem \ref{em} embeds  $\Quot_{\pi^*\F}$ into the product of Grasmmann bundles 
\[\iota: \Quot_{\pi^*\F} \hookrightarrow G(2, \F^{\oplus 2}) \times_U G(4, \F^{\oplus 3}),\]
where the universal bundle $S_1$ (resp. $S_2$) restricts to $p_*\mathcal{S}(1)$ (resp. $p_*\mathcal{S}(2)$).
Moreover,
\[[\Quot_{\pi^*\F}] = c_{top}(S^\vee_{1} \otimes Q_2)^2 = \left(\Delta^2_{23}\left[\frac{c(Q_2)}{c(S_1)} \right]\right)^2 = \left(\Delta^2_{23} \left[ \frac{c(\rho^*\F)^3}{c(S_1)c(S_2)} \right]\right)^2. \]

On $\Quot_{\pi^*\F}$, the locus where $\mathcal{S}$ has splitting type $(-2,0)$ is the same as where $\mathcal{S}(1)$ has splitting type $(-1,1)$. This is given by the universal rank $2$ formula:
\[[Z_{(-2,0)}]= \left[ \frac{c((p_*\mathcal{S}(1))^\vee)^2}{c((p_*\mathcal{S}(2))^\vee)}\right]_1 = c_1(p_*\mathcal{S}(2)) - 2c_1(p_*\mathcal{S}(1)) = \iota^*(c_1(S_2) - 2c_1(S_1)).\]
Then \eqref{topush} says
\[[\overline{\Sigma}_{(-2,0,2)}] = \rho_*\left( \left(\Delta^2_{23}\left[\frac{c(\rho^*\F)^3}{c(S_1)c(S_2)} \right]\right)^2  \cdot \Delta^6_6 \left[c(\rho^*\G) \frac{c(S_1)^5}{c(S_2)^4} \right] \cdot (c_1(S_2) - 2c_1(S_1))\right).\]
The class inside the outer parenthesis is codimension $129$ and the relative fiber dimension of $\rho$ is $124$, so the pushforward is a codimension $5$ class on the base.

\subsubsection*{Explicit computation} Let $f_i = c_i(\mathcal{F}) = c_i((\pi_*E(m))^\vee)$ and $g_i = c_i(\mathcal{G}) =  c_i((\pi_*E(m-1))^\vee)$.
To implement the algorithm, we first stored all push forwards of monomials in the Chern classes of $S_1$ and $S_2$, as determined by \cite[Cor. 2.6]{HT}. This precomputation took 5 days on 6 cores.
Then, we expanded the above class as a polynomial in $\rho^* f_i, \rho^* g_i$ and the Chern classes of $S_1$ and $S_2$ and computed the push forward. The second step took 2 days on 6 cores and produced the following formula:
\begin{align*}
[\overline{\Sigma}_{(-2,0,2)}] = &  \ 4f_1^4g_1 - 8f_1^3g_1^2 + 4f_1^2g_1^3 - 3f_1^3f_2 - 6f_1^2f_2g_1 + 13f_1f_2g_1^2 - 4f_2g_1^3 + 8f_1^2g_1g_2 \\
&- 8f_1g_1^2g_2 + 6f_1f_2^2 + 3f_1^2f_3 - 2f_2^2g_1 + 2f_1f_3g_1 - 5f_3g_1^2 - 6f_1f_2g_2 - 2f_2g_1g_2 \\
&+ 4g_1g_2^2 - 8f_1g_1g_3 + 8g_1^2g_3 - 6f_2f_3 - 3f_1f_4 + 2f_4g_1 + 6f_3g_2 + 6f_2g_3 - 6g_2g_3 \\
&+ 2g_1g_4 + 3f_5 - 6g_5.
\end{align*}
Sage code for these processes may be found at \url{http://web.stanford.edu/~hlarson/}.

\end{Example}

\end{document}